\documentclass[11pt]{amsart}

\usepackage{latexsym}
\usepackage{amsmath}
\usepackage{amssymb}

\usepackage{epsfig}

\usepackage{enumerate}

\newtheorem{theorem}{Theorem}[section]
\newtheorem{lemma}[theorem]{Lemma}

\newcommand{\cN}{{\mathcal N}}
\newcommand{\cP}{{\mathcal P}}
\newcommand{\cT}{{\mathcal T}}

\newcommand{\lsa}{{\rm lsa}}

\usepackage{color}
\newcommand{\blue}{\textcolor{black}}

\usepackage{tikz}
\usepackage{subcaption}

\begin{document}
	
\title{Trinets Encode Orchard Phylogenetic Networks}
	
\thanks{The first author was supported by the New Zealand Marsden Fund} 
	
\author{Charles Semple}
\address{School of Mathematics and Statistics, University of Canterbury, Christchurch, New Zealand}
\email{charles.semple@canterbury.ac.nz}
	
\author{Gerry Toft}
\address{School of Mathematics and Statistics, University of Canterbury, Christchurch, New Zealand}
\email{gerry.toft@pg.canterbury.ac.nz}
	
	
	
\date{\today}
	
\begin{abstract}
Rooted triples, rooted binary phylogenetic trees on three leaves, are sufficient to encode rooted binary phylogenetic trees. That is, if $\cT$ and $\cT'$ are rooted binary phylogenetic $X$-trees \blue{that infers} the same set of rooted triples, then $\cT$ and $\cT'$ are isomorphic. However, in general, this sufficiency does not extend to rooted binary phylogenetic networks. In this paper, we show that trinets, phylogenetic network analogues of rooted triples, are sufficient to encode rooted binary orchard networks. Rooted binary orchard networks naturally generalise rooted binary tree-child networks. Moreover, we present a polynomial-time algorithm for building a rooted binary orchard network from its set of trinets. As a consequence, this algorithm affirmatively answers a previously-posed question of whether there is a polynomial-time algorithm for building a rooted binary tree-child network from the set of trinets it infers.
%We also give a counterexample to a conjecture concerning the encoding of ``recoverable'' rooted phylogenetic networks.
\end{abstract}

\maketitle

\section{Introduction}

The evolutionary relationships of a collection of present-day species are typically represented by a rooted phylogenetic (evolutionary) tree. Over recent decades, a wide variety of methods for building rooted phylogenetic trees from genomic data have been developed~\cite{fel04} and these methods are routinely used by computational biologists. However, it is now well recognised~\cite{hus10} that, as the result of non-treelike (reticulate) processes, rooted phylogenetic networks provide a more accurate representation of the evolutionary relationships for many such collections. These processes include hybridisation and lateral gene transfer. Consequently, a central current task in computational biology is the development of methods for building rooted phylogenetic networks.

A canonical and practical approach to building rooted phylogenetic networks is to amalgamate smaller networks (or trees) on overlapping leaf sets into a single rooted phylogenetic network. In the context of building rooted phylogenetic trees, which amalgamate smaller trees, these approaches are collectively called supertree methods and they have been very successful in the inference of rooted phylogenetic trees (for example, see~\cite{bin04}). A desirable property of any supertree method is that if the smaller trees are consistent, then the returned supertree infers each of the smaller trees. The theorem that underlies this property is the following. Loosely speaking, a set $\cP$ of (smaller) rooted phylogenetic networks ``encodes'' a rooted phylogenetic network $\cN$ if $\cN$ is the only rooted phylogenetic network to infer each of the networks in $\cP$. A rooted \blue{binary} phylogenetic tree is encoded by the set of all rooted triples it infers, \blue{where a rooted triple is a rooted binary phylogenetic tree on three leaves} (see, for example, \cite{aho81, sem03}). In this paper, we are interested in analogues of this theorem for phylogenetic networks.

With a necessary mild restriction, Gambette and Huber~\cite{gam12} showed that rooted binary level-$1$ networks, that is, rooted binary phylogenetic networks whose underlying cycles are vertex disjoint, are encoded by the set of all rooted triples they infer. However, this result does not generalise to rooted binary level-$2$ networks~\cite{gam12}. On the other hand, generalising level-$1$ networks in a different direction, Linz and Semple~\cite{lin20} showed recently that rooted binary normal networks~\cite{wil10}, while not encoded by the set of rooted triples they infer, are encoded by the set of all rooted binary caterpillars on three and four leaves they infer. This improves upon a result of Willson~\cite{wil11} who showed that a rooted binary normal network on $n$ leaves is encoded by the set of all rooted phylogenetic trees on $n$ leaves it infers. Note that a rooted caterpillar on three leaves is the same as a rooted triple. However, analogous results for rooted binary tree-child networks~\cite{car09}, a slight generalisation of rooted binary normal networks, do not hold. In particular, rooted binary tree-child networks on $n$ leaves are not necessarily encoded by the set of all rooted binary phylogenetic trees on $n$ leaves they infer (for \blue{an} example, see~\cite{lin20}). Thus, if we want to build the correct rooted phylogenetic network using a supertree-type approach, doing so with trees is limiting.

As a consequence of the work in~\cite{gam12}, Huber and Moulton~\cite{hub13} considered building rooted phylogenetic networks from smaller networks, in particular, rooted phylogenetic networks on three leaves which they called trinets. In contrast to above, van Iersel and Moulton~\cite{ier14} showed that rooted binary level-$2$ networks as well as rooted binary tree-child networks are encoded by the set of all trinets they infer. In this paper, we generalise this result for tree-child networks to the recently introduced class of rooted binary orchard networks~\cite{erd19, jan18}. That is, we show that a rooted binary orchard network is encoded by the set of trinets it infers. Unlike rooted binary tree-child networks whose total number of vertices is bounded linearly in the size of its leaf set~\cite{car09}, for a fixed set of leaves, the total number of vertices in a rooted binary orchard network can be arbitrarily large. Also, for any positive integer $k$, there exists a rooted binary orchard network whose level is at least~$k$. Furthermore, we present a polynomial-time algorithm for \blue{building} a rooted binary orchard networks from the set of trinets it infers. \blue{As a consequence,} this answers a question of van Iersel and Moulton~\cite{ier14} of whether it is possible to build a rooted binary tree-child network from the set of trinets it infers in polynomial time.
%and also give a counterexample to a conjecture concerning the encoding of ``recoverable'' rooted phylogenetic networks.
We next formally state the main result of the paper.
%and the conjecture.

Throughout the paper, $X$ denotes a non-empty finite set and all paths are directed.

\noindent{\bf Phylogenetic networks.} A {\em binary phylogenetic network $\cN$ on $X$} is a rooted acyclic directed graph with no arcs in parallel satisfying the following properties:
\begin{enumerate}[{\rm (i)}]
\item the (unique) root has in-degree zero and out-degree two;
		
\item the set of vertices with out-degree zero is $X$ and all such vertices have in-degree one;
		
\item all other vertices either have in-degree one and out-degree two, or in-degree two and out-degree one.
\end{enumerate}
Additionally, if $|X|=1$, we allow $\cN$ to consist of the single vertex in $X$. The vertices in $X$ are called {\em leaves}, and so we refer to $X$ as the {\em leaf set} of $\cN$. Furthermore, vertices of in-degree one and out-degree two are {\em tree vertices}, while vertices of in-degree two and out-degree one are {\em reticulations}. Arcs directed into a reticulation are called {\em reticulation arcs}. If $\cN$ has no reticulations, then $\cN$ is a {\em rooted binary phylogenetic $X$-tree}. Since we only consider rooted binary phylogenetic trees and binary phylogenetic networks, we will abbreviate such trees and networks to rooted phylogenetic trees and phylogenetic networks, respectively. To illustrate, a phylogenetic network $\cN_1$ on $\{x_1, x_2, \ldots, x_6\}$ is shown in Fig.~\ref{orchard}. In this figure, as in all other figures in the paper, all arcs are directed down the page.

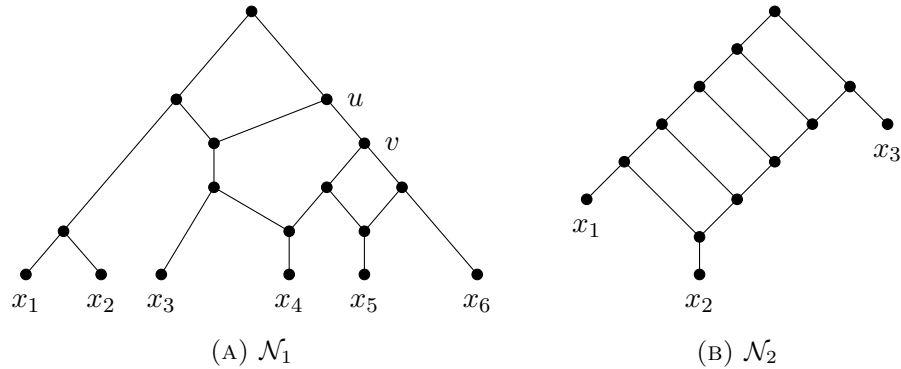
\begin{figure}
\centering
\begin{subfigure} {0.6\textwidth}
\centering
\begin{tikzpicture}
\draw (0,0) -- (-3,-3.5);
\draw (0,0) -- (3,-3.5);
\draw (-2.5,-2.925) -- (-2,-3.5);
\draw (-1,-1.17) -- (-0.5,-1.755);
\draw (1,-1.17) -- (-0.5,-1.755) -- (-0.5,-2.34) -- (-1.2,-3.5);
\draw (1.5,-1.755) -- (0.5,-2.925) -- (0.5,-3.5);
\draw (-0.5,-2.34) -- (0.5,-2.925);
\draw (1,-2.34) -- (1.5,-2.925) -- (1.5,-3.5);
\draw (2,-2.34) -- (1.5,-2.925);

\filldraw[black] (0,0) circle (2pt);
\filldraw[black] (-3,-3.5) circle (2pt) node [label=below:$x_1$] {};
\filldraw[black] (3,-3.5) circle (2pt) node [label=below:$x_6$] {};
\filldraw[black] (-2.5,-2.925) circle (2pt);
\filldraw[black] (-2,-3.5) circle (2pt) node [label=below:$x_2$] {};
\filldraw[black] (-1,-1.17) circle (2pt);
\filldraw[black] (-0.5,-1.755) circle (2pt);
\filldraw[black] (1,-1.17) circle (2pt) node [label=right:$u$] {};
\filldraw[black] (-0.5,-2.34) circle (2pt);
\filldraw[black] (-1.2,-3.5) circle (2pt) node [label=below:$x_3$] {};
\filldraw[black] (1.5,-1.755) circle (2pt) node [label=right:$v$] {};
\filldraw[black] (0.5,-2.925) circle (2pt);
\filldraw[black] (0.5,-3.5) circle (2pt) node [label=below:$x_4$] {};
\filldraw[black] (1,-2.34) circle (2pt);
\filldraw[black] (1.5,-2.925) circle (2pt);
\filldraw[black] (1.5,-3.5) circle (2pt) node [label=below:$x_5$] {};
\filldraw[black] (2,-2.34) circle (2pt);
\end{tikzpicture}
\subcaption{$\cN_1$}
\end{subfigure}
\hfill%
\begin{subfigure} {0.38\textwidth} 
\centering	
\begin{tikzpicture}
\draw (0,0) -- (-2.5,-2.5);
\draw (0,0) -- (1.5,-1.5);
\draw (1,-1) -- (-1,-3) -- (-1,-3.5);
\draw (-0.5,-0.5) -- (0.5,-1.5);
\draw (-1,-1) -- (0,-2);
\draw (-1.5,-1.5) -- (-0.5,-2.5);
\draw (-2,-2) -- (-1,-3);

\filldraw[black] (0,0) circle (2pt);
\filldraw[black] (-0.5,-0.5) circle (2pt);
\filldraw[black] (-1,-1) circle (2pt);
\filldraw[black] (-1.5,-1.5) circle (2pt);
\filldraw[black] (-2,-2) circle (2pt);
\filldraw[black] (-2.5,-2.5) circle (2pt) node [label=below:$x_1$] {};
\filldraw[black] (-1,-3.5) circle (2pt) node [label=below:$x_2$] {};
\filldraw[black] (-1,-3) circle (2pt);
\filldraw[black] (-0.5,-2.5) circle (2pt);
\filldraw[black] (0,-2) circle (2pt);
\filldraw[black] (0.5,-1.5) circle (2pt);
\filldraw[black] (1,-1) circle (2pt);
\filldraw[black] (1.5,-1.5) circle (2pt) node [label=below:$x_3$] {};
\end{tikzpicture}
\subcaption{$\cN_2$}
\end{subfigure}
\caption{A phylogenetic network $\cN_1$ on $\{x_1, x_2, \ldots, x_6\}$ and a phylogenetic network $\cN_2$ on $\{x_1, x_2, x_3\}$. \blue{In $\cN_1$, both $u$ and $v$ are stable ancestors of $\{x_5, x_6\}$.}}
\label{orchard}
\end{figure}

Let $\cN$ be a phylogenetic network on $X$. If $u$ and $v$ are vertices of $\cN$ and there is a path from $u$ to $v$, we say $u$ is an {\em ancestor} of $v$ or, equivalently, $v$ is a {\em descendant} of $u$. Note that every vertex is an ancestor, and thus a descendant, of itself. Furthermore, if $|X|\ge 2$, for each leaf $x\in X$, we denote the (unique) parent of $x$ by $p_x$.

Let $\cN_1$ and $\cN_2$ be two phylogenetic networks on $X$ with vertex and arc sets $V_1$ and $E_1$, and $V_2$ and $E_2$, respectively. We say $\cN_1$ is {\em isomorphic} to $\cN_2$ if there exists a bijection $\varphi: V_1\rightarrow V_2$ such that $\varphi(x)=x$ for all $x\in X$, and $(u, v)\in E_1$ if and only if $(\varphi(u), \varphi(v))\in E_2$ for all $u, v\in V_1$.

\noindent {\bf Orchard networks.} Let $\cN$ be a phylogenetic network on $X$. Let $\{a, b\}$ be a $2$-element subset of $X$. We say $\{a, b\}$ is a {\em cherry} if $p_a=p_b$. Furthermore, $(a, b)$ is a {\em reticulated cherry} if $(p_a, p_b)$ is a \blue{reticulation} arc of $\cN$ where, necessarily, $p_a$ is a tree vertex and $p_b$ is a reticulation. The arc $(p_a, p_b)$ is called the {\em reticulation arc} of $(a, b)$. As an example, consider the phylogenetic network $\cN_1$ shown in Fig.~\ref{orchard}. The set $\{x_1, x_2\}$ is a cherry, while $(x_3, x_4)$ and $(x_6, x_5)$ are reticulated cherries of $\cN_1$. We next describe two reduction operations associated with cherries and reticulated cherries. First, suppose that $\{a, b\}$ is a cherry of $\cN$. Let $\cN'$ be the phylogenetic network on $X-\{b\}$ obtained from $\cN$ by deleting $b$ and its incident arc, and suppressing the resulting degree-two vertex $p_a$. We say that $\cN'$ has been obtained from $\cN$ by {\em reducing $b$}. Note that if $p_a$ is the root of $\cN$, the operation of reducing $b$ corresponds to replacing $\cN$ with the phylogenetic network consisting of the single vertex $a$. Second, suppose that $(a, b)$ is a reticulated cherry of $\cN$. Now let $\cN'$ be the phylogenetic network on $X$ obtained from $\cN$ by deleting $(p_a, p_b)$ and suppressing the two resulting degree-two vertices $p_a$ and $p_b$. We say $\cN'$ has been obtained from $\cN$ by {\em cutting $(a, b)$}. For ease of reading, we sometimes refer to these operations as {\em picking a cherry} or {\em picking a reticulated cherry}, respectively.

A phylogenetic network $\cN$ is {\em orchard} if there is a sequence
$$\cN=\cN_0, \cN_1, \cN_2, \ldots, \cN_k$$
of phylogenetic networks such that, for each $i\in \{1, 2, \ldots, k\}$, the phylogenetic network $\cN_i$ is obtained from $\cN_{i-1}$ by either reducing a leaf of a cherry or cutting a reticulated cherry, and $\cN_k$ consists of a single vertex. It is easily checked that both $\cN_1$ and $\cN_2$ in Fig.~\ref{orchard} are orchard networks. For $\cN_2$, we can obtain a sequence by repeatedly cutting the reticulated cherry $(x_1, x_2)$ until there are no more reticulations, and then reducing $x_3$ of the cherry $\{x_2, x_3\}$, and reducing $x_2$ of the cherry $\{x_1, x_2\}$. It may appear that the order in which we pick a cherry or a reticulated cherry is important, but this is not the case as the following lemma~\cite{erd19, jan18} shows.

\begin{lemma}
Let $\cN$ be an orchard network, and suppose that $\cN'$ is obtained from $\cN$ by picking either a cherry or a reticulated cherry. Then $\cN'$ is an orchard network.
\label{order}
\end{lemma}

Orchard networks were introduced independently in~\cite{erd19} and~\cite{jan18}, and generalise the more familiar class of tree-child networks. A phylogenetic network is {\em tree-child} if every non-leaf vertex is the parent of a tree vertex or a leaf~\cite{car09}. However, not all phylogenetic networks are orchard. For example, neither of the two phylogenetic networks shown in Fig.~\ref{counterexample} is orchard.

\noindent {\bf Trinets.} A {\em trinet} is a phylogenetic network on three leaves. Observe that trinets generalise the more familiar concept of {\em rooted triples}, rooted \blue{(binary)} phylogenetic trees on three leaves.
	
Let $\cN$ be a phylogenetic network on $X$, and let $X'$ be a subset of $X$. A {\em stable ancestor} of $X'$ is a vertex $u$ of $\cN$ having the property that, for all $x\in X'$, every path from the root of $\cN$ to $x$ traverses $u$. In the literature, a stable ancestor of $X'$ is also referred to as a ``visible'' ancestor of $X'$. Since the root itself satisfies this property, such a vertex always exists. Furthermore, we say $u$ is a {\em lowest stable ancestor} of $X'$ if no distinct stable ancestor of $X'$ is a descendant of $u$. Note that if $u$ and $v$ are stable ancestors of $X'$, then there is either a path from $u$ to $v$, or a path from $v$ to $u$. It follows that the lowest stable ancestor of $X'$ is unique. We denote the lowest stable ancestor of $X'$ by $\lsa(X')$. In Fig.~\ref{orchard}, $u$ and $v$ are stable ancestors of $\{x_5, x_6\}$ in $\cN_1$, but $v$ is the lowest stable ancestor of $\{x_5, x_6\}$ in $\cN_1$.

For a directed graph $G$, the {\em full simplification} of $G$ is the directed graph obtained from $G$ by repeatedly suppressing vertices of in-degree one and out-degree one, and deleting exactly one arc of any pair of arcs in parallel until neither of these operations are applicable. Now, let $\cN$ be a phylogenetic network on $X$, and let $X'$ be a subset of $X$. Suppose that $u$ is the lowest stable ancestor of $X'$. The {\em path graph of $\cN$ on $X'$} is the directed subgraph of $\cN$ obtained by deleting all vertices and arcs not on a path from $u$ to a leaf in $X'$. That is, the path graph of $\cN$ on $X'$ consists of all paths \blue{of $\cN$} starting at $u$ and ending at a vertex in $X'$. The phylogenetic network {\em exhibited by $\cN$ on $X'$} is the full simplification of the path graph of $\cN$ on $X'$. We denote the phylogenetic network exhibited by $\cN$ on $X'$ by $\cN_{X'}$. In the special case $|X'|=3$, this process constructs the {\em trinet exhibited by $\cN$ on $X'$}. The set of all trinets exhibited by $\cN$ is denoted by $Tn(\cN)$. Again consider the phylogenetic network $\cN_1$ shown in Fig.~\ref{orchard}. Noting that the root is the lowest stable ancestor of $\{x_2, x_3, x_4\}$, the path graph of $\cN_1$ on $\{x_2, x_3, x_4\}$ is shown in Fig.~\ref{trinet}(A), while the full simplification of this path graph, that is, the trinet exhibited by $\cN_1$ on $\{x_2, x_3, x_4\}$, is shown in Fig.~\ref{trinet}(B).

\begin{figure}
\centering
\begin{subfigure} {0.4\textwidth}
\centering
\begin{tikzpicture}
\draw (0,0) -- (-2.5,-2.925);
\draw (0,0) -- (1.5,-1.755);
\draw (-2.5,-2.925) -- (-2,-3.5);
\draw (-1,-1.17) -- (-0.5,-1.755);
\draw (1,-1.17) -- (-0.5,-1.755) -- (-0.5,-2.34) -- (-1.2,-3.5);
\draw (1.5,-1.755) -- (0.5,-2.925) -- (0.5,-3.5);
\draw (-0.5,-2.34) -- (0.5,-2.925);

\filldraw[black] (0,0) circle (2pt);
\filldraw[black] (-2.5,-2.925) circle (2pt);
\filldraw[black] (-2,-3.5) circle (2pt) node [label=below:$x_2$] {};
\filldraw[black] (-1,-1.17) circle (2pt);
\filldraw[black] (-0.5,-1.755) circle (2pt);
\filldraw[black] (1,-1.17) circle (2pt);
\filldraw[black] (-0.5,-2.34) circle (2pt);
\filldraw[black] (-1.2,-3.5) circle (2pt) node [label=below:$x_3$] {};
\filldraw[black] (1.5,-1.755) circle (2pt);
\filldraw[black] (0.5,-2.925) circle (2pt);
\filldraw[black] (0.5,-3.5) circle (2pt) node [label=below:$x_4$] {};
\filldraw[black] (1,-2.34) circle (2pt);
\end{tikzpicture}
\subcaption{The path-graph exhibited by $\cN_1$ on $\{x_2,x_3,x_4\}$.}
\end{subfigure}
\hspace{0.05\textwidth}%
\begin{subfigure} {0.4\textwidth}
\centering
\begin{tikzpicture}
\draw (0,0) -- (-1,-1.17);
\draw (0,0) -- (1,-1.17);
\draw (-1,-1.17) -- (-2,-3.5);
\draw (-1,-1.17) -- (-0.5,-1.755);
\draw (1,-1.17) -- (-0.5,-1.755) -- (-0.5,-2.34) -- (-1.2,-3.5);
\draw (1,-1.17) -- (0.5,-2.925) -- (0.5,-3.5);
\draw (-0.5,-2.34) -- (0.5,-2.925);

\filldraw[black] (0,0) circle (2pt);
\filldraw[black] (-2,-3.5) circle (2pt) node [label=below:$x_2$] {};
\filldraw[black] (-1,-1.17) circle (2pt);
\filldraw[black] (-0.5,-1.755) circle (2pt);
\filldraw[black] (1,-1.17) circle (2pt);
\filldraw[black] (-0.5,-2.34) circle (2pt);
\filldraw[black] (-1.2,-3.5) circle (2pt) node [label=below:$x_3$] {};
\filldraw[black] (0.5,-2.925) circle (2pt);
\filldraw[black] (0.5,-3.5) circle (2pt) node [label=below:$x_4$] {};
\end{tikzpicture}
\subcaption{The trinet exhibited by $\cN_1$ on $\{x_2,x_3,x_4\}$.}
\end{subfigure}
\caption{The path graph of the phylogenetic network $\cN_1$, shown in Fig.~\ref{orchard}, on $\{x_2, x_3, x_4\}$, and the trinet exhibited by $\cN_1$ on $\{x_2, x_3, x_4\}$.}
\label{trinet}
\end{figure}
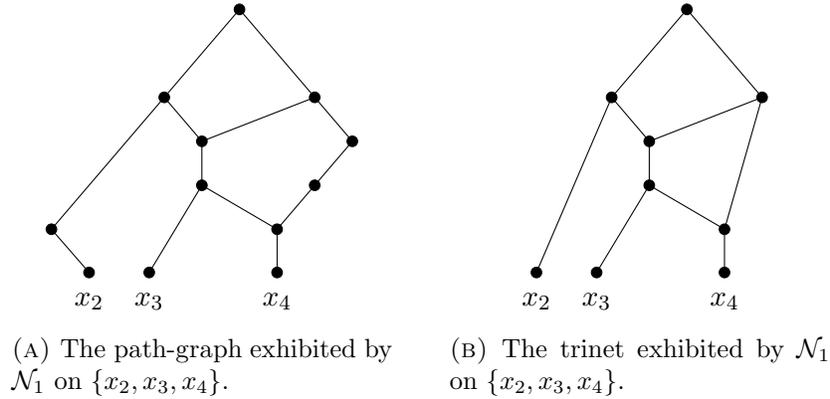

A phylogenetic network $\cN$ on $X$ is {\em recoverable} if it has no arc $(u, v)$ whose deletion disconnects $\cN$ and $v$ is an ancestor of every element in $X$, that is, $v$ is a stable ancestor of $X$ (and $v$ is not the root). Equivalently, $\cN$ is recoverable if $\lsa(X)$ is the root of $\cN$. We say a recoverable phylogenetic network $\cN$ is {\em encoded} by $Tn(\cN)$ if it has the following property: If $\cN'$ is a recoverable phylogenetic network and, \blue{up to isomorphism,} $Tn(\cN)=Tn(\cN')$, then $\cN$ is isomorphic to $\cN'$. Observe that if a phylogenetic network $\cN$ is not recoverable, then $Tn(\cN)$ provides no information of the structure of $\cN$ between the root of $\cN$ and an arc $(u, v)$ whose deletion disconnects $\cN$ and in which every leaf is descendant of $v$.

The next theorem is one of two main results in~\cite{ier14}. It generalises the well-known result mentioned earlier that says a rooted phylogenetic tree $\cT$ is encoded by the set of all rooted triples exhibited by $\cT$ (see, for example, \cite{aho81, sem03}). All tree-child networks are recoverable since, provided the leaf set has size at least two, the root has out-degree two and every non-leaf vertex is the parent of a tree vertex or a leaf.

\begin{theorem}
Let $\cN$ be a tree-child network on $X$, \blue{where $|X|\ge 3$}. Then $Tn(\cN)$ encodes $\cN$.
\label{tree-child}
\end{theorem}

The \blue{first part of Theorem~\ref{main}}, the main result of this paper, generalises Theorem~\ref{tree-child} to the class of orchard networks. The second part of Theorem~\ref{main} shows that orchard networks can be reconstructed from the set of trinets they exhibit in polynomial time which, as a consequence, answers a question of~\cite{ier14} \blue{of whether such a reconstruction is possible for tree-child networks}.

\begin{theorem}
Let $\cN$ be an orchard network on $X$, where $|X|\ge 3$. Then
\begin{enumerate}[{\rm (i)}]
\item $Tn(\cN)$ encodes $\cN$, and

\item up to isomorphism, $\cN$ can be reconstructed from $Tn(\cN)$ in time $O(|V|^6)$, where $V$ is the vertex set of $\cN$.
\end{enumerate}
\label{main}
\end{theorem}

As we show in the next section, like tree-child networks, orchard networks are recoverable. However, unlike tree-child networks whose total number of reticulations is at most linear in the size of their leaf sets, and so the total number of vertices in a tree-child network is bounded (see~\cite{mcd15}), the total number of reticulations in an orchard network is not bounded by the size of its leaf set. For example, by extending $\cN_2$ in Fig.~\ref{orchard} in the obvious way, it follows that, even with three leaves, the total number of reticulations in an orchard network is not bounded. Moreover, this extension also shows that, for each non-negative integer $k$, there exists an orchard network whose level is at least $k$. A phylogenetic network is {\em level-$k$} if each biconnected component contains at most $k$ reticulations.

In addition to Theorem~\ref{tree-child}, the second main result in~\cite{ier14} establishes that recoverable level-$2$ phylogenetic networks are also encoded by their sets of exhibited trinets. These results, together with Theorem~\ref{main}, support the conjecture in~\cite{hub13}, and restated in~\cite{ier14}, that if a phylogenetic network $\cN$ is recoverable, then $Tn(\cN)$ encodes $\cN$. However, Huber et al.~\cite{hub15} construct a family of counterexamples to this conjecture, where the level of the phylogenetic network is exponential in the size of the leaf set. In particular, for all $n\ge 4$, if the size of the leaf set is $n$, then the level of the phylogenetic network is $(2^{n-2}-1)n$. Thus if $n=4$, then the level of the counterexample is~$12$. This raises the problem of determining the largest value of~$k$ for which all recoverable level-$k$ phylogenetic networks are encoded by their sets of trinets. In the last section of the paper, we show that this value is at most~$3$ by showing that the two non-isomorphic recoverable phylogenetic networks $\cN_1$ and $\cN_2$, each of level-$4$, shown in Fig.~\ref{counterexample} have the property that, up to isomorphism, $Tn(\cN_1)=Tn(\cN_2)$. Note that each of the counterexamples $\cN$ in~\cite{hub15} have the much stronger property that the set of all phylogenetic networks exhibited by $\cN$ on all proper subsets of the leaf set of $\cN$ does not encode $\cN$.

\begin{figure}
		\centering
		\begin{subfigure} {0.45\textwidth}
			\centering
			\begin{tikzpicture}
			\draw (0,0) -- (-2.4,-3);
			\draw (0,0) -- (2.4,-3);
			\draw (1.2,-1.5) -- (-0.8,-3);
			\draw (-1.2,-1.5) -- (-0.06,-2.355);
			\draw (0.06,-2.445) -- (0.8,-3);
			\draw (-2.4,-3) -- (-1.6,-4) -- (-0.8,-3) -- (0,-4) -- (0.8, -3) -- (1.6, -4) -- (2.4,-3);
			\draw (-2.4,-3) -- (-3.2,-4) -- (-3.2,-5);
			\draw (-1.6,-4) -- (-1.6,-5);
			\draw (0,-4) -- (0,-5);
			\draw (1.6,-4) -- (1.6,-5);
			
			\filldraw[white] (-1.805, -3.75) circle (2pt);
			\filldraw[white] (-1.32, -3.67) circle (2pt);
			\filldraw[white] (-0.41, -3.51) circle (2pt);
			\filldraw[white] (0.53, -3.33) circle (2pt);
			\filldraw[white] (0.98, -3.25) circle (2pt);
			\draw (2.4,-3) -- (-3.2,-4);
			
			\filldraw[black] (0,0) circle (2pt);
			\filldraw[black] (-2.4,-3) circle (2pt);
			\filldraw[black] (1.2,-1.5) circle (2pt);
			\filldraw[black] (-0.8,-3) circle (2pt);
			\filldraw[black] (-1.2,-1.5) circle (2pt);
			\filldraw[black] (0.8,-3) circle (2pt);
			\filldraw[black] (-2.4,-3) circle (2pt);
			\filldraw[black] (-1.6,-4) circle (2pt);
			\filldraw[black] (-0.8,-3) circle (2pt);
			\filldraw[black] (0,-4) circle (2pt);
			\filldraw[black] (0.8,-3) circle (2pt);
			\filldraw[black] (1.6,-4) circle (2pt);
			\filldraw[black] (2.4,-3) circle (2pt);
			\filldraw[black] (0,-5) circle (2pt) node [label=below:$x_3$] {};
			\filldraw[black] (-1.6,-5) circle (2pt) node[label=below:$x_2$] {};
			\filldraw[black] (1.6,-5) circle (2pt) node [label=below:$x_4$] {};
			\filldraw[black] (-3.2,-4) circle (2pt);
			\filldraw[black] (-3.2,-5) circle (2pt) node[label=below:$x_1$] {};
			
			\draw (-0.6,-0.75) -- (0, -1.35);
			\filldraw[black] (-0.6,-0.75) circle (2pt);
			\filldraw[black] (0,-1.35) circle (2pt) node [label=below:$x_5$] {};
			\end{tikzpicture}
			\subcaption{$\cN_1$}
		\end{subfigure}
		\begin{subfigure} {0.45\textwidth}
			\centering
			\begin{tikzpicture}
			\draw (0,0) -- (-2.4,-3);
			\draw (0,0) -- (2.4,-3);
			\draw (1.2,-1.5) -- (-0.8,-3);
			\draw (-1.2,-1.5) -- (-0.06,-2.355);
			\draw (0.06,-2.445) -- (0.8,-3);
			\draw (-2.4,-3) -- (-1.6,-4) -- (-0.8,-3) -- (0,-4) -- (0.8, -3) -- (1.6, -4) -- (2.4,-3);
			\draw (-2.4,-3) -- (-3.2,-4) -- (-3.2,-5);
			\draw (-1.6,-4) -- (-1.6,-5);
			\draw (0,-4) -- (0,-5);
			\draw (1.6,-4) -- (1.6,-5);
			
			\filldraw[white] (-1.805, -3.75) circle (2pt);
			\filldraw[white] (-1.32, -3.67) circle (2pt);
			\filldraw[white] (-0.41, -3.51) circle (2pt);
			\filldraw[white] (0.53, -3.33) circle (2pt);
			\filldraw[white] (0.98, -3.25) circle (2pt);
			\draw (2.4,-3) -- (-3.2,-4);
			
			\filldraw[black] (0,0) circle (2pt);
			\filldraw[black] (-2.4,-3) circle (2pt);
			\filldraw[black] (1.2,-1.5) circle (2pt);
			\filldraw[black] (-0.8,-3) circle (2pt);
			\filldraw[black] (-1.2,-1.5) circle (2pt);
			\filldraw[black] (0.8,-3) circle (2pt);
			\filldraw[black] (-2.4,-3) circle (2pt);
			\filldraw[black] (-1.6,-4) circle (2pt);
			\filldraw[black] (-0.8,-3) circle (2pt);
			\filldraw[black] (0,-4) circle (2pt);
			\filldraw[black] (0.8,-3) circle (2pt);
			\filldraw[black] (1.6,-4) circle (2pt);
			\filldraw[black] (2.4,-3) circle (2pt);
			\filldraw[black] (0,-5) circle (2pt) node [label=below:$x_3$] {};
			\filldraw[black] (-1.6,-5) circle (2pt) node[label=below:$x_2$] {};
			\filldraw[black] (1.6,-5) circle (2pt) node [label=below:$x_4$] {};
			\filldraw[black] (-3.2,-4) circle (2pt);
			\filldraw[black] (-3.2,-5) circle (2pt) node[label=below:$x_1$] {};
			
			\draw (0.6,-0.75) -- (0, -1.35);
			\filldraw[black] (0.6,-0.75) circle (2pt);
			\filldraw[black] (0,-1.35) circle (2pt) node [label=below:$x_5$] {};
			\end{tikzpicture}
			\subcaption{$\cN_2$}
		\end{subfigure}
		\caption{Two non-isomorphic \blue{level-$4$} phylogenetic networks $\cN_1$ and $\cN_2$. Both $\cN_1$ and $\cN_2$ are recoverable and, up to isomorphism, $Tn(\cN_1)=Tn(\cN_2)$.}
		\label{counterexample}
\end{figure}
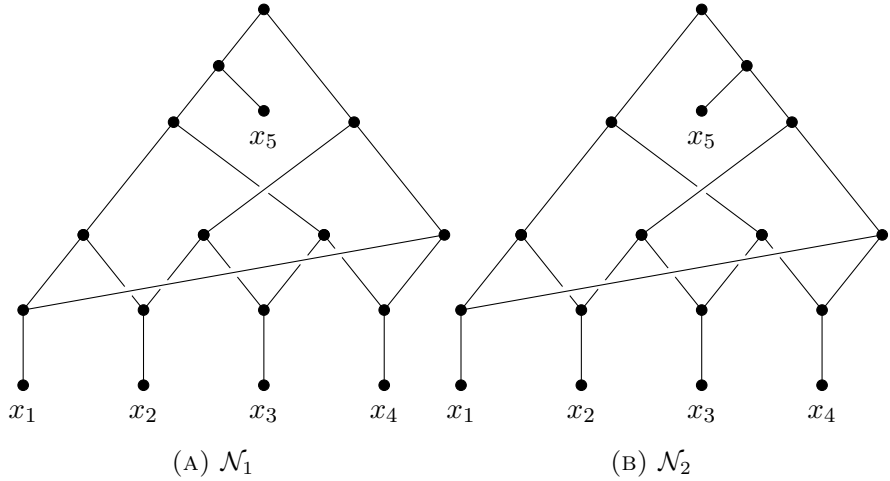

The paper is organised as follows. The next section consists of some preliminary lemmas which are used in the proof of Theorem~\ref{main}. The proof of Theorem~\ref{main} is by induction on the sum of the number of leaves and the number of reticulations \blue{of an orchard network}. The approach taken is to initially pick either a cherry, thereby reducing the number of leaves, or a reticulated cherry, thereby reducing the number of reticulations. In Section~\ref{exhibit2}, we establish various lemmas concerning the notion of exhibit and that of cherries and reticulated cherries. The proof of Theorem~\ref{main} is given in Section~\ref{proof}. The last section, Section~\ref{example}, verifies the above-mentioned level-$4$ example.

\blue{We end the introduction with two remarks. First, a concept in mathematical phylogenetics that is similar to exhibit is that of ``display''. In particular, let $\cN$ be a phylogenetic network on $X$ and let $\cT$ be a rooted phylogenetic $X'$-tree, where $X'\subseteq X$. We say $\cN$ {\em displays} $\cT$ if $\cT$ can be obtained from $\cN$ by deleting arcs and vertices, and suppressing any resulting vertices of in-degree one and out-degree one. If $\cN$ is a rooted phylogenetic tree, then the concepts of exhibit and display are equivalent. For clarification, in the initial part of the introduction, whenever we said, for example, a phylogenetic network ``infers'' a rooted phylogenetic tree, we really meant a phylogenetic network displays a rooted phylogenetic tree.}

Second, Theorem~\ref{main} and other analogous theorems are a step towards developing supertree-type methods for building phylogenetic networks. In practice, it is unlikely that the input to such a method is the entire set $Tn(\cN)$ of trinets \blue{exhibited by} a phylogenetic network $\cN$. A more realistic task is when the input is an arbitrary subset of trinets and the goal is to decide whether or not there is a phylogenetic network that exhibits each of the trinets in this set. This has been considered previously for when the input is a set of rooted triples and we are asked to find a level-$1$ network that \blue{displays} each of the rooted triples in the set~\cite{ier11, jan06a, jan06b} and, more recently, when the input is a set of trinets and we are asked to find a level-$1$ network that exhibits each of the trinets in the set~\cite{hub17}. As an intermediate step towards developing a supertree-type method for building orchard networks, we leave it as an open problem to develop an algorithm that takes an arbitrary collection of trinets on overlapping leaf sets and decides whether or not there is an orchard network that exhibits each trinet in the collection.

\section{Exhibiting Lemmas}
\label{exhibit1}

In this section, we establish some general lemmas in relation to the notion of exhibiting that will used in the proof of Theorem~\ref{main}. The first two lemmas are used in several places.

\begin{lemma}
Let $\cN$ be a phylogenetic network on $X$, and let $A\subseteq X$. Let $G_A$ be the path graph of $\cN$ on $A$, and let
$$G_A=G_0, G_1, G_2, \ldots, G_k=\cN_A$$
be a sequence of directed graphs such that, for all $i\in \{1, 2, \ldots, k\}$, the directed graph $G_i$ is obtained from $G_{i-1}$ by ether suppressing a vertex of in-degree one and out-degree one, or deleting an arc in parallel. \blue{Let $u$ and $v$ be vertices of $G_i$ for some $i$. Then}
\begin{enumerate}[{\rm (i)}]
\item \blue{If there is a path from $u$ to $v$ in $G_i$, then there is a path from $u$ to $v$ in $G_A$.}

\item If $u$ and $v$ are vertices of $G_i$ for some $i$, then every path from $u$ to a (fixed) leaf $\ell$ traverses $v$ in $G_A$ if and only if every path from $u$ to $\ell$ traverses $v$ in $G_i$.
\end{enumerate}
\label{traverse1}
\end{lemma}

\begin{proof}
\blue{We omit the proof of (i) as it takes the same approach as the proof of (ii) but is simpler. For the proof of (ii),} it suffices to show that if $j\in \{0, 1, \ldots, i-1\}$, then every path from $u$ to $\ell$ traverses $v$ in $G_j$ if and only if every path from $u$ to $\ell$ traverses $v$ in $G_{j+1}$. Clearly, this sufficiency holds if $G_{j+1}$ is obtained from $G_j$ by deleting an arc in parallel. Therefore assume that $G_{j+1}$ is obtained from $G_j$ by suppressing a vertex, say $w$, of in-degree one and out-degree one. Let $e$ denote the new arc in $G_{j+1}$ resulting from this suppression. Now, if $P$ is a path from $u$ to $\ell$ that traverses $v$ and $w$ in $G_j$, then the path obtained from $P$ by replacing $w$ and its incident arcs with $e$ is a path from $u$ to $\ell$ that traverses $v$ and $e$ in $G_{j+1}$. Since the analogous converse of this also holds, it follows that every path from $u$ to $\ell$ traverses $v$ in $G_j$ if and only if every path from $u$ to $\ell$ traverses $v$ in $G_{j+1}$. This completes the proof of the lemma.
\end{proof}

\begin{lemma}
Let $\cN$ be a phylogenetic network on $X$, and let $A\subseteq X$. Let $u$ and $v$ be vertices of the path graph $G_A$ of $\cN$ on $A$, and let $\ell\in A$. \blue{Then}
\begin{enumerate}[{\rm (i)}]
\item If every path from $u$ to $\ell$ traverses $v$ in $G_A$, then every path from $u$ to $\ell$ traverses $v$ in $\cN$.

\item If $v$ is a stable ancestor of $\ell$ in $\cN_A$, then $v$ is a stable ancestor of $\ell$ in $\cN$.
\end{enumerate}
\label{traverse2}
\end{lemma}

\begin{proof}
Since $u$ and $v$ are vertices of $G_A$, the proof of (i) is an immediate consequence of the construction of $G_A$ from $\cN$. To prove (ii), suppose that $v$ is a stable ancestor of $\ell$ in $\cN_A$. Then, as the root of $\cN_A$ is $\lsa(A)$, it follows by Lemma~\ref{traverse1}, that every path from $\lsa(A)$ to $\ell$ in $G_A$ traverses $v$. As every path from the root of $\cN$ to $\ell$ traverses $\lsa(A)$, it follows by (i) that $v$ is a stable ancestor of $\ell$ in $\cN$.
\end{proof}

The next three lemmas provide sufficient conditions for a vertex of a phylogenetic network $\cN$ to be a vertex of the phylogenetic network exhibited by $\cN$ on a given subset of leaves.

\begin{lemma}
\label{split_desc_lemma}
Let $\cN$ be a phylogenetic network on $X$, and let $A\subseteq X$. Let $v$ be a tree vertex of $\cN$ with children $c_1$ and $c_2$, and suppose there exists $\ell_1, \ell_2 \in A$ such that
\begin{enumerate}[{\rm (i)}]
\item $\ell_1$ is a descendant of $c_1$,
\item $\ell_2$ is a descendant of $c_2$, and
\item $\ell_2$ is not a descendant of $c_1$.
\end{enumerate}
Then $v$ is a vertex of $\cN_A$.
\end{lemma}

\begin{proof}
We first show that $v$ is a vertex of the path graph $G_A$ of $\cN$ on $A$. Since there is a path \blue{in $\cN$} from $v$ to a leaf in $A$, either $v$ is a descendant of \blue{$\lsa(A)$} or \blue{$\lsa(A)$} is a descendant of $v$. If the latter holds, then there are paths from $c_1$ to $\ell_1$ and from $c_2$ to $\ell_2$, each of which traverses \blue{$\lsa(A)$}. This implies that there is a path from $c_1$ to $\ell_2$ via \blue{$\lsa(A)$}, contradicting (iii). Hence $v$ is a descendant of \blue{$\lsa(A)$}, and so $v$ is a vertex of $G_A$.

We complete the proof of the lemma by showing that $v$ is not suppressed in the process of obtaining $\cN_A$ from $G_A$. If $v$ is suppressed, then at some stage in the process, $v$ has one incoming arc and one outgoing arc, $(v, w)$ say. By Lemma~\ref{traverse1}, every path in $G_A$ from $v$ to a leaf in $A$ traverses $w$ which, in turn, implies by Lemma~\ref{traverse2} that every path in $\cN$ from $v$ to a leaf in $A$ traverses $w$. In particular, every path in $\cN$ from $c_1$ to $\ell_1$ and from $c_2$ to $\ell_2$ traverses $w$, in which case, there is a path in $\cN$ from $c_1$ to $\ell_2$ via $w$, contradicting (iii). It follows that $v$ is a vertex of $\cN_A$.
\end{proof}

\begin{lemma}
\label{retic_parent_lemma}
Let $\cN$ be a phylogenetic network on $X$, and let $A\subseteq X$. Let $v$ be a reticulation of $\cN$. If a parent of $v$ is a vertex of $\cN_A$, then $v$ is a vertex of $\cN_A$.
\end{lemma}

\begin{proof}
\blue{Let $p$ and $q$ denote the parents of $v$ in $\cN$, and suppose that $p$ is a vertex of $\cN_A$.} We begin by showing that $v$, \blue{as well as $p$ and $q$}, is a vertex of the path graph $G_A$ of $\cN$ on $A$. \blue{Now} $p$ lies on a path of $\cN$ from $\lsa(A)$ to a leaf in $A$. If $p$ is a reticulation of $\cN$, then $v$ also lies on this path. Furthermore, if $p$ is a tree vertex of $\cN$, then, as $p$ is a vertex of $\cN_A$, both children of $p$ must also lie on such a path; otherwise, $p$ has in-degree one and out-degree one in $G_A$. Thus $v$ is a vertex of $G_A$, and so both parents of $v$ are \blue{also} vertices of $G_A$.

It remains to show that $v$ is not suppressed in the process of obtaining $\cN_A$ from $G_A$. If $v$ is suppressed in this process, then, as $p$ is a vertex of $\cN_A$, at subsequent stages in the process of obtaining $\cN_A$ from $G_A$, the \blue{vertex $q$} is suppressed, and $v$ has two distinct incoming arcs in parallel, one of which is $(p, v)$. Since $p$ is a vertex of $\cN_A$, this in turn implies that the other arc in parallel also connects $p$ and $v$. But then $p$ is a tree vertex of $\cN$ and so, once one of these parallel arcs is deleted, $p$ has in-degree one and out-degree one, a contradiction as $p$ is a vertex of $\cN_A$. Hence $v$ is not suppressed in obtaining $\cN_A$ from $G_A$, and so $v$ is a vertex of $\cN_A$.
\end{proof}

\begin{lemma}
\label{all_desc_lemma}
Let $\cN$ be a phylogenetic network on $X$, and let $A\subseteq X$. Let $v$ be a tree vertex of $\cN$ that is a descendant of $\lsa(A)$. If $A$ contains every leaf of $\cN$ that is a descendant of $v$, then every descendant vertex of $v$ \blue{in $\cN$} is a vertex of $\cN_A$.
\end{lemma}

\begin{proof}
First observe that every vertex that is a descendant of $v$ is a vertex of the path graph $G_A$ of $\cN$ on $A$. Furthermore, as $A$ contains every leaf that is a descendant of $v$, it follows that no vertex that is a descendant of $v$ has in-degree one and out-degree one in $G_A$. Suppose that at some stage of the process of obtaining $\cN_A$ from $G_A$ a descendant of $v$, say $w$, has in-degree one and out-degree one. Without loss of generality, choose $w$ to be the first such descendant of $v$ with this property. If $w$ is a tree vertex of $\cN$, then $w$ has two distinct children and so, for $w$ to have in-degree one and out-degree one, one if its children needs to have in-degree one and out-degree one prior to this happening. \blue{As} both children of $w$ are descendants of $v$, \blue{this contradicts} the choice of $w$. Thus we may assume that $w$ is a reticulation of $\cN$.

At least one parent, $p$ say, of $w$ is a descendant of $v$ \blue{in $\cN$}. Since $w$ is suppressed in the process of obtaining $\cN_A$ from $G_A$, it follows by the choice of $w$ that, in this process prior to $w$ having in-degree one and out-degree one, the parent of $w$ that is not $p$, say $q$, is suppressed and $w$ has two distinct incoming arcs in parallel, one of which is $(p, w)$. \blue{By Lemma~\ref{traverse1}(i), this implies that $q$ is a descendant of $v$ in $G_A$, and so} $q$ is a descendant of $v$ \blue{in $\cN$}, contradicting the choice of $w$. This completes the proof of the lemma.
\end{proof}

\begin{lemma}
\label{subset_phynet_lemma}
Let $\mathcal N$ be a phylogenetic network on $X$, and let $A\subseteq B\subseteq X$. Then $\cN_A$ is the phylogenetic network exhibited by $\cN_B$ on $A$.
\end{lemma}

\begin{proof}
Let $G_A$ and $G_B$ be the path graphs of $\mathcal N$ on $A$ and $B$, respectively. Since $A\subseteq B$, the vertex \blue{$\lsa(A)$} is a descendant of \blue{$\lsa(B)$}, and so \blue{$\lsa(A)$} is a vertex of $G_B$. In turn, this implies that $G_A$ is a subgraph of $G_B$. If $v$ is a vertex of $G_A$, then the in-degree of $v$ in $G_A$ is at most the in-degree of $v$ in $G_B$, and the out-degree of $v$ in $G_A$ is at most the out-degree of $v$ in $G_B$. Therefore, every vertex of $G_A$ that is not a vertex of $\cN_B$ is also not a vertex of $\cN_A$. Thus the directed graph $G'_A$, the path graph of $\cN_B$ on $A$, can be obtained from $G_A$ by repeated applications of suppressing vertices of in-degree one and out-degree one, and deleting exactly one arc of any pair of arcs in parallel. Note that we need not take the full simplification of $G_A$ to get $G'_A$. \blue{Since} $\cN_A$ is the full simplification of $G'_A$, \blue{it follows that $\cN_A$ is the phylogenetic network exhibited by $\cN_B$ on $A$}.\end{proof}

The last lemma of this section uses each of Lemmas~\ref{split_desc_lemma}--\ref{subset_phynet_lemma} in its proof.

\begin{lemma}
\label{sup_ret_cherry}
Let $\cN$ be a phylogenetic network on $X$, \blue{where $|X|\ge 3$}, and let $(a, b)$ be a reticulated cherry of $\cN$. Let $p_b$ denote the parent of $b$, and let $A\subseteq X$ such that $A=\{b, x, y\}$. Then $p_b$ is a vertex of $\cN_A$ if and only if $p_b$ is a vertex of at least one of $\cN_{\{b, x\}}$ and $\cN_{\{b, y\}}$.
\end{lemma}

\begin{proof}
First suppose that $p_b$ is a vertex of $\cN_A$. Let $v$ be a tree vertex (possibly the root) of $\cN_A$ with the property that there is a path \blue{$P$} from $v$ to $p_b$ such that every vertex on this path (except $v$ itself) is a reticulation. Note that such a vertex can be found by starting at $p_b$ and moving along reticulation arcs towards the root of $\cN_A$. If neither $x$ nor $y$ is a descendant of $v$, then, by Lemmas~\ref{all_desc_lemma} and~\ref{subset_phynet_lemma}, $p_b$ is a vertex of both $\cN_{\{b, x\}}$ and $\cN_{\{b, y\}}$. Therefore, without loss of generality, we may assume $x$ is descendant of $v$ in $\cN_A$. \blue{Let $w$ denote} the first reticulation along \blue{$P$, and note that $w$ could be $p_b$. Since} the only leaf descendant of \blue{$w$} is $b$, \blue{it follows} by Lemma~\ref{split_desc_lemma} \blue{that} $v$ is a vertex of $\cN_{\{b, x\}}$. By repeated applications of Lemma~\ref{retic_parent_lemma} \blue{to the reticulations along $P$, we deduce that} $p_b$ is also a vertex of $\cN_{\{b, x\}}$.

Now suppose that $p_b$ is a vertex of $\cN_{\{b, z\}}$, where $z\in \{x, y\}$. By Lemma~\ref{subset_phynet_lemma}, the phylogenetic network $\cN_{\{b, z\}}$ is the phylogenetic network exhibited by $\cN_A$ on $\{b, z\}$, and so $p_b$ is a vertex of $\cN_A$.
\end{proof}

\section{Cherry and Reticulated-Cherry Lemmas}
\label{exhibit2}

The lemmas in this section are more aligned with orchard networks. We begin by showing that orchard networks are recoverable.

\begin{lemma}
\label{orchard_recoverable}
Let $\mathcal N$ be an orchard network on $X$. Then $\mathcal N$ is recoverable.
\end{lemma}

\begin{proof}
Let $\rho$ denote the root of $\cN$. The proof is by induction on the sum of the number $n=|X|$ of leaves and the number $r$ of reticulations \blue{of $\cN$}. If $n+r=1$, then $\cN$ has exactly one leaf and no reticulations. Thus $\cN$ consists of the single vertex in $X$, and so the lemma holds. If $n+r=2$, then, as $\cN$ is orchard, $\cN$ consists of two leaves adjoined to $\rho$. Again, the lemma holds.
		
Now suppose that $n+r\ge 3$, and that every orchard network in which the sum of the number of leaves and the number of reticulations is at most $n+r-1$ is recoverable. Let $\cN'$ be a phylogenetic network on $X'$ that is obtained from $\cN$ by picking either a cherry $\{a, b\}$ or a reticulated cherry $(a, b)$. Note that the roots of $\cN$ and $\cN'$ coincide as $\cN$ does not consist of two leaves adjoined to the root. By Lemma~\ref{order}, $\cN'$ is orchard. Therefore, as the sum of the number of leaves and number of reticulations \blue{of $\cN'$} is $n+r-1$, it follows by the induction assumption that $\cN'$ is recoverable. That is, the root of $\cN'$ is the unique stable ancestor of $X'$ in $\cN'$. As the roots of $\cN$ and $\cN'$ coincide, up to traversing $p_a$ and $p_b$ (the parents of $a$ and $b$, respectively, in $\cN$), every path in $\cN'$ from the root to a leaf $x$ in $X'$ is also a path in $\cN$ from $\rho$ to $x$. It follows that $\rho$ is the unique stable ancestor of $X$ in $\cN$, and so $\cN$ is recoverable.
\end{proof}

\begin{lemma}
\label{cherry_pick_recoverable}
Let $\mathcal N$ be a (arbitrary) recoverable phylogenetic network, and suppose that $\mathcal N'$ is obtained from $\cN$ by picking \blue{either} a cherry or a reticulated cherry. Then $\mathcal N'$ is recoverable.
\end{lemma}

\begin{proof}
Let $X'$ denote the leaf set of $\cN'$, and let $\{a, b\}$ or $(a, b)$ be the cherry or reticulated cherry of $\cN$ that is picked to obtain $\cN'$. \blue{Observe that we may assume the roots of $\cN$ and $\cN'$ coincide; otherwise, $\cN'$ consists of a single vertex and the lemma holds.} Suppose that $\mathcal N'$ is not recoverable. Then there is a non-root vertex $v'$ of $\cN'$ that is a stable ancestor of $X'$. Consider $v'$ in $\mathcal N$. Since $\mathcal N$ is recoverable, there must be a path $P$ from the root of $\cN$ to a leaf that does not traverse $v'$. As $\cN'$ is obtained from $\cN$ by picking either $\{a, b\}$ or $(a, b)$, this path $P$ must end at $b$. It follows that $\cN'$ is obtained from $\cN$ by picking $(a, b)$ and $P$ traverses $(p_a, p_b)$. But this implies there is a path in $\cN'$ from the root of $\cN'$ to $a$ that does not traverse $v'$, contradicting that \blue{$v'$ is a stable ancestor of $X'$}. Hence $\mathcal N'$ is recoverable.
%Suppose $(a,b)$ is a cherry, and let $p$ be the parent of $a$ and $b$. Every path from the root to $b$ passes through $p$. But this implies there is a path from the root to $p$ which does not pass through $v$, a contradiction, since every path from the root to $a$ passes through $v$. Now suppose $(a,b)$ is a reticulated cherry. A path to $b$ either passes through the parent of $a$, or it does not. The former case is a contradiction, since every path to $a$ passes through $v$, so every path to the parent of $a$ also passes through $v$. The latter case is also a contradiction, since this path is also a path to a leaf in $\mathcal N'$, so passes through $v$.
\end{proof}

Let $\cN$ be a phylogenetic network on $X$, and let $\{a, b\}\subseteq X$. If $\{a, b\}$ is a cherry of $\cN$, we refer to $p_a$ as the {\em tree vertex of $\{a, b\}$}, while if $(a, b)$ is a reticulated cherry, we refer to $p_a$ as the {\em tree vertex of $(a, b)$}.

\begin{lemma}
\label{trinet_cherry_lemma}
Let $\cN$ be a phylogenetic network on $X$, \blue{where $|X|\ge 3$}, and let $\{a, b\}\subseteq X$. Then
\begin{enumerate}[{\rm (i)}]
\item $\{a, b\}$ is a cherry of $\cN$ if and only if, for all $A$ with $\{a, b\}\subseteq A\subseteq X$ and $|A|=3$, we have that $\{a, b\}$ is a cherry of $\cN_A$, and

\item $(a, b)$ is a reticulated cherry of $\cN$ if and only if, for all $A$ with $\{a, b\}\subseteq A\subseteq X$ and $|A|=3$, we have that $(a, b)$ is a reticulated cherry of $\cN_A$.
\end{enumerate}
\end{lemma}

\begin{proof}
\blue{We will prove (ii). The proof of (i) closely follows the proof of (ii) and is omitted. Let $A\subseteq X$ such that $\{a, b\}\subseteq A$ and $|A|=3$.}
%If $\{a, b\}$ is a cherry of $\cN$. Then $p_a$ satisfies the conditions of Lemma~\ref{all_desc_lemma}, and so $a$, $b$, and $p_a$ are all vertices of $\cN_A$. Thus $\{a, b\}$ is a cherry of $\cN_A$. Similarly,
If $(a, b)$ is a reticulated cherry of $\cN$, then $p_a$ satisfies the conditions of Lemma~\ref{all_desc_lemma}. Thus $a$, $b$, $p_a$, and $p_b$ are all vertices of $\cN_A$, and so $(a, b)$ is a reticulated cherry of $\cN_A$.

For the converse of (ii), suppose that $(a, b)$ is not a reticulated cherry of $\cN$. We will show that there is a trinet exhibited by $\cN$ whose leaf set contains $a$ and $b$, but $(a, b)$ is not a reticulated cherry of this trinet. \blue{If there is no trinet exhibited by $\cN$ in which $(a, b)$ is a reticulated cherry, then the desired outcome holds. So we may assume that there exists a trinet $\cN_A$ exhibited by $\cN$ in which} $(a, b)$ is a reticulated cherry. Let $u$ be the tree vertex of $(a, b)$ of $\cN_A$. In $\cN$, the vertex $u$ is a tree vertex in which $a$ is a descendant. Since $u$ is stable ancestor of $a$ in $\cN_A$, it follows by Lemma~\ref{traverse2} that every path from the root of $\cN$ to $a$ traverses $u$. Thus, if $(a, b)$ is a reticulated cherry of another trinet exhibited by $\cN$ and $u'$ is the tree vertex of $(a, b)$ \blue{of this trinet}, then $u$ is either an ancestor or a descendant of $u'$ in $\cN$. It now follows that there is a path $P$ in $\cN$ from the root of $\cN$ to $a$ containing every vertex that is the tree vertex of $(a, b)$ of a trinet exhibited by $\cN$ in which $(a, b)$ is a reticulated cherry.

Let $v$ denote the last such tree vertex along $P$. If $a$ and $b$ are the only leaf descendants of $v$ in $\cN$, then, by Lemma~\ref{all_desc_lemma}, for any choice of $A$ containing $a$ and $b$, all descendant vertices of $v$ \blue{in $\cN$} are vertices of the trinet exhibited by $\cN$ on $A$. But $(a, b)$ is not a reticulated cherry of $\cN$, so $v$ is not the tree vertex of any trinet exhibited by $\cN$ in which $(a, b)$ is a reticulated cherry, a contradiction. Therefore, \blue{in $\cN$, the vertex} $v$ has a leaf descendant, say $\ell$, other than $a$ and $b$. Consider the trinet exhibited by $\cN$ on $\{a, b, \ell\}$. If $v$ is not a vertex of the path graph of $\cN$ on $\{a, b, \ell\}$, then $\lsa(\{a, b, \ell\})$ is a descendant of $v$ in $\cN$, and so, by the choice of $v$, the \blue{ordered pair} $(a, b)$ is not a reticulated cherry of $\cN_{\{a, b, \ell\}}$, and we have the desired outcome. Thus we may assume that $v$ is a vertex of the path graph of $\cN$ on $\{a, b, \ell\}$. If $v$ is a vertex of $\cN_{\{a, b, \ell\}}$, then $a$, $b$, and $\ell$ are descendants of $v$ in $\cN_{\{a, b, \ell\}}$. Therefore, if $(a, b)$ is a reticulated cherry of $\cN_{\{a, b, \ell\}}$, then $v$ is not its tree vertex. But every other possible such tree vertex is an ancestor of $v$ in $\cN$. Hence, $(a, b)$ is not a reticulated cherry of $\cN_{\{a, b, \ell\}}$. The final case to consider is when $v$ is suppressed in the process of obtaining $\cN_{\{a, b, \ell\}}$ from the path graph of $\cN$ on $\{a, b, \ell\}$. Then the unique child of $v$ in this step of the process or a descendant of \blue{this} child is a vertex of $\cN_{\{a, b, \ell\}}$ and has $a$, $b$, and $\ell$ as descendants. But every vertex which is a tree vertex of $(a, b)$ in some trinet exhibited by $\cN$ is an ancestor of this descendant of $v$, so $(a, b)$ is not a reticulated cherry of $\cN_{\{a, b, \ell\}}$. This completes the proof of the converse of (ii), and thus the lemma.
\end{proof}

\begin{lemma}
\label{cherry_construct_trinet}
Let $\cN$ be a phylogenetic network on $X$, and let $\{a, b\}$ be a cherry of $\cN$. Let $\cN'$ be the phylogenetic network obtained from $\cN$ by reducing $b$, and suppose that $A\subseteq X-\{b\}$. Then $\cN_A = \cN'_A$.
\end{lemma}

\begin{proof}
First observe that $\lsa(A)$ of $\cN$ is also $\lsa(A)$ of $\cN'$. Clearly the lemma holds if $|A|=1$, so we may assume that $|A|\ge 2$. Let $G_A$ be the path graph of $\cN$ on $A$, and let $G'_A$ be the path graph of $\cN'$ on $A$. Let $\ell\in A$, where $\ell\neq a$. Then every path in $\cN$ from $\lsa(A)$ to $\ell$ is a path in $\cN'$ from $\lsa(A)$ to $\ell$. Therefore, if $a\not\in A$, the path graph $G'_A$ is identical to $G_A$, and so $\cN_A=\cN'_A$. On the other hand, if $a\in A$, then every path in $\cN$ from $\lsa(A)$ to $a$ traverses $p_a$, and so suppressing $p_a$ in such a path produces a path in $\cN'$ from $\lsa(A)$ to $a$. Moreover, all paths in $\cN'$ from $\lsa(A)$ to $a$ can be obtained in this way. Thus, in $G_A$, the vertex $p_a$ has in-degree one and out-degree one, and so $G'_A$ is obtained from $G_A$ by suppressing $p_a$. \blue{It follows that} $\cN_A=\cN'_A$.
\end{proof}

\begin{lemma}
\label{ret_construct_trinet}
Let $\cN$ be a phylogenetic network on $X$, and let $(a, b)$ be a reticulated cherry of $\cN$. Let $p_a$ and $p_b$ denote the parents of $a$ and $b$, respectively, in $\cN$. Let $\cN'$ be the phylogenetic network obtained from $\cN$ by cutting $(a, b)$, and suppose that $A\subseteq X$. Then each of the following holds:
\begin{enumerate}[{\rm (i)}]
\item If $b\not\in A$, then $\cN_A=\cN'_A$.

\item If $a, b\in A$, then $\cN'_A$ is obtained from $\cN_A$ by deleting $(p_a, p_b)$ and suppressing $p_a$ and $p_b$.

\item If $a\not\in A$, $b\in A$, and $p_b$ is not a vertex of $\cN_A$, then $\cN_A=\cN'_A$.

\item If $a\not\in A$, $b\in A$, and $p_b$ is a vertex of $\cN_A$, then $\cN'_A$ is obtained from $\cN_A$ by
\begin{enumerate}[{\rm (I)}]
\item deleting the arc $(u, p_b)$, where $u$ is a vertex such that there is a path in $\cN$ from $u$ to $p_b$ traversing $p_a$, and every non-terminal vertex along this path is not a vertex of $\cN_A$,

\item repeatedly deleting non-leaf vertices of out-degree zero until there are no such vertices, and

\item taking the full simplification of the resulting directed graph.
\end{enumerate}
\end{enumerate}
\end{lemma}

\begin{proof}
Let $A\subseteq X$. Let $G_A$ be the path graph of $\cN$ on $A$, and let $G'_A$ be the path graph of $\cN'$ on $A$. If $a, b\not\in A$, then $G_A=G'_A$, so $\cN_A=\cN'_A$. If $a\in A$ and $b\not\in A$, then, up to suppressing $p_a$, we have $G_A=G'_A$. Thus $\cN_A=\cN'_A$. Therefore~(i) holds and so, for the remainder of the proof, we may assume that $b\in A$, in which case $(p_a, p_b)$ is an arc of $G_A$.

Let $G^0_A = G_A$, and let $H^0_A$ be the directed graph obtained from $G_A$ by deleting $(p_a, p_b)$. Note that if $a\in A$, then $G'_A$ can be obtained from $H^0_A$ by suppressing $p_a$ and $p_b$. Furthermore, if $a\not\in A$, then $G'_A$ can be obtained from $H^0_A$ by \blue{suppressing $p_b$ and} repeatedly deleting non-leaf vertices with out-degree zero.

Suppose that $u_0$ is a vertex of $G^0_A$ with in-degree one and out-degree one, but $u_0\neq p_a$. Note that $u_0\neq p_b$. In constructing $H^0_A$ from $G^0_A$, the only vertices whose degrees changed were $p_a$ and $p_b$. Therefore, $u_0$ also has in-degree one and out-degree one in $H^0_A$. Construct $G^1_A$ and $H^1_A$ from $G^0_A$ and $H^0_A$, respectively, by suppressing $u_0$ and deleting exactly one arc of any resulting pair of parallel arcs. Observe that if an arc in parallel is deleted, then it is not incident with $p_a$ or $p_b$. Furthermore, $H^1_A$ can be obtained from $G^1_A$ by deleting $(p_a, p_b)$, and that $\cN'_A$ can be obtained from $H^1_A$ by repeatedly deleting any non-leaf vertices with out-degree zero until there are no such vertices, and then taking the full simplification of the resulting directed graph.

Now iteratively repeat this process. That is, for $i\ge 1$, suppose that $u_i$ is a vertex of $G^i_A$ with in-degree one and out-degree one, but $u_i\neq p_a$. Construct $G^{i+1}_A$ and $H^{i+1}_A$ from $G^i_A$ and $H^i_A$, respectively, by suppressing $u_i$ and deleting exactly one arc of any resulting pair of parallel arcs. In general, $H^i_A$ can be obtained from $G^i_A$ by deleting $(p_a, p_b)$, and $\cN'_A$ can be obtained from $H^i_A$ by repeatedly deleting any non-leaf vertices with out-degree zero until there are no such vertices, and then taking the full simplification of the resulting directed graph. Eventually, after, say $k$, iterations, we construct $G^k_A$ and $H^k_A$ where, except possibility $p_a$, there is no vertex of $G^k_A$ with in-degree one and out-degree one.

If $a\in A$, then $p_a$ does not have in-degree one and out-degree one in $G^k_A$, so $G^k_A$ has no vertices of in-degree one and out-degree one \blue{(and thus, no pair of parallel arcs)}. Therefore $G^k_A=\cN_A$. Thus, as $H^k_A$ is obtained from $G^k_A$ by deleting $(p_a, p_b)$ and $a\in A$, it follows that $\cN'_A$ is obtained from $\cN_A$ by deleting $(p_a, p_b)$ and suppressing $p_a$ and $p_b$. Hence~(ii) holds. Therefore we may now assume $a\not\in A$.

Since $a\not\in A$, the vertex $p_a$ has in-degree one and out-degree one in $G^k_A$, and $p_a$ has out-degree zero in $H^k_A$. Let $p$ be the parent of $p_a$ in $G^k_A$. Construct $G^{k+1}_A$ from $G^k_A$ by suppressing $p_a$, and construct $H^{k+1}_A$ from $H^k_A$ by deleting $p_a$. Observe that $H^{k+1}_A$ can be obtained from $G^{k+1}_A$ by deleting $(p, p_b)$. If $G^{k+1}_A=\cN_A$, then $p_b$ is a vertex of $G^{k+1}_A$, and $H^{k+1}_A$ can be obtained from $\cN_A$ by deleting $(p, p_b)$. Therefore, $\cN'_A$ can be obtained from $\cN_A$ by deleting $(p, p_b)$, repeatedly deleting non-leaf vertices of out-degree zero until there are no such vertices, and then taking the full simplification of the resulting directed graph. Thus~(iv) holds.

If $G^{k+1}_A\neq \cN_A$, then $G^{k+1}_A$ has either a vertex of in-degree one and out-degree one, or a pair of parallel arcs. By the construction of $G^{k+1}_A$, the only possibility is that $p$ has a pair of outgoing parallel arcs to $p_b$. In this case, $\cN_A$ is obtained from $G^{k+1}_A$ by deleting one of these arcs to $p_b$ and suppressing $p$ and $p_b$. Since $H^{k+1}_A$ is obtained from $G^{k+1}_A$ by deleting $(p, p_b)$, it follows that $\cN_A$ is obtained from $H^{k+1}_A$ by suppressing $p$ and $p_b$. Hence $\cN_A=\cN'_A$, thereby establishing~(iii) and completing the proof of the lemma.
\end{proof}

\section{Proof of Theorem~\ref{main}}
\label{proof}

This section consists of the proof of Theorem~\ref{main}. We begin by first establishing Theorem~\ref{main}(i).

\begin{proof}[Proof of Theorem~\ref{main}(i)]
Let $\cN$ be an orchard network on $X$, where $|X|\ge 3$, and let $\cN_0$ be a recoverable phylogenetic network on $X$ such that, up to isomorphism, $Tn(\cN_0) = Tn(\cN)$. The proof is by induction on the sum of the number $n$ of leaves and the number $r$ of reticulations of $\cN$. If $r=0$, then $\cN$ is a phylogenetic tree and $Tn(\cN)$ consists of all rooted triples exhibited by $\cN$. Thus, by~\cite[Theorem~6.4.1]{sem03}, Theorem~\ref{main}(i) holds.
%If $n+r=1$, then $\cN$ as well as $\cN_0$ consists of the single vertex in $X$, and so $\cN\cong \cN_0$. If $n+r=2$, then, as $\cN$ is orchard, $\cN$ has no reticulations and so $n=2$. Thus $\cN$, as well as $\cN_0$, consists of two leaves adjoined to the root. Again, $\cN\cong \cN_0$.
Furthermore, if $n=3$, then $\cN$ exhibits exactly one trinet. By Lemma~\ref{orchard_recoverable}, orchard networks are recoverable, and so $\lsa(X)$ is the root of $\cN$. Therefore this trinet is $\cN$ itself. Since, up to isomorphism, $Tn(\cN)=Tn(\cN_0)$ and $\cN_0$ is recoverable, it follows that $\cN\cong \cN_0$.

Now suppose that $n\ge 4$ and $r\ge 1$, so $n+r\ge 5$, and that the theorem holds for all orchard networks in which the sum of the number of leaves and the number of reticulations is at most $n+r-1$. Since $\cN$ is orchard, $\cN$ has either a cherry, say $\{a, b\}$, or a reticulated cherry, say $(a, b)$. \blue{Up to isomorphism,} $Tn(\cN)=Tn(\cN_0)$ \blue{and so, by Lemma~\ref{trinet_cherry_lemma}}, either $\{a, b\}$ is a cherry or $(a, b)$ is a reticulated cherry of $\cN_0$, respectively. Let $\cN'$ and $\cN_0'$ be the phylogenetic networks obtained from $\cN$ and $\cN_0$, respectively, by reducing $b$ or cutting $(a, b)$. By Lemma~\ref{order}, $\cN'$ is orchard and, by Lemma~\ref{cherry_pick_recoverable}, $\cN_0'$ is recoverable.

First suppose that $\{a, b\}$ is a cherry of $\cN$ and $\cN_0$. By Lemma~\ref{cherry_construct_trinet}, $Tn(\cN')$ and $Tn(\cN'_0)$ are obtained from $Tn(\cN)$ and $Tn(\cN_0)$, respectively, by excluding those trinets whose leaf set contains $b$. Therefore, up to isomorphism, as $Tn(\cN)=Tn(\cN_0)$, \blue{we have} $Tn(\cN')=Tn(\cN_0')$. Thus, by the induction assumption, $\cN'\cong \cN_0'$. Since $\{a, b\}$ is a cherry of $\cN$ and $\cN_0$, we deduce that $\cN\cong \cN_0$.

% so to reattach $b$, we must subdivide the arc into $a$ with a new vertex $p_a$, and add an arc $(p_a, b)$.  This is the only way to reattach $b$ (while preserving the trinets), so $\cN = \cN_0$.

Now suppose that $(a, b)$ is a reticulated cherry of $\cN$ and $\cN_0$. We will use Lemma~\ref{ret_construct_trinet} to show that the trinets \blue{exhibited by} $\cN'$ can be determined from the trinets \blue{exhibited by} $\cN$. The same argument will also show that the trinets \blue{exhibited by} $\cN'_0$ can be determined from the trinets \blue{exhibited by} $\cN_0$ in the same way. Noting that the leaf set of $\cN'$ is $X$, let $A\subseteq X$, where $|A|=3$. If $b\not\in A$ or $a, b\in A$, then we can construct $\cN'_A$ from $\cN_A$ as described by Lemma~\ref{ret_construct_trinet}(i) and (ii), respectively. Thus we may assume that $b\in A$, but $a\not\in A$. Say $A=\{b, x, y\}$, where $a\not\in \{x, y\}$. Let $p_a$ and $p_b$ denote the parents of $a$ and $b$, respectively, in $\cN$. We next use the trinets exhibited by $\cN$ to determine whether or not $p_b$ is a vertex of $\cN_A$.

Consider $\cN_{\{a, b, x\}}$. By Lemma~\ref{trinet_cherry_lemma}, $(a, b)$ is a reticulated cherry of $\cN_{\{a, b, x\}}$, and so $p_b$ is a vertex of $\cN_{\{a, b, x\}}$. By Lemma~\ref{subset_phynet_lemma}, the phylogenetic network exhibited by $\cN$ on $\{b, x\}$ is also the phylogenetic network exhibited by $\cN_{\{a, b, x\}}$ on $\{b, x\}$. Thus we can construct $\cN_{\{b, x\}}$ from $\cN_{\{a, b, x\}}$. In particular, we can decide whether or not $p_b$ is a vertex of $\cN_{\{b, x\}}$ from $\cN_{\{a, b, x\}}$. Similarly, we can decide whether or not $p_b$ is a vertex of $\cN_{\{b, y\}}$ from $\cN_{\{a, b, y\}}$. If $p_b$ is a vertex of neither $\cN_{\{b, x\}}$ nor $\cN_{\{b, y\}}$, then, by Lemma~\ref{sup_ret_cherry}, $p_b$ is not a vertex of $\cN_A$, and so, by Lemma~\ref{ret_construct_trinet}(iii), $\cN'_A=\cN_A$. Therefore, we may assume there exists $z\in \{x, y\}$ such that $p_b$ is a vertex of $\cN_{\{b, z\}}$, \blue{in which case, by Lemma~\ref{sup_ret_cherry}}, $p_b$ is a vertex of $\cN_A$. Let $p_1$ and $p_2$ be the parents of $p_b$ in $\cN_A$. Recalling that $\cN'$ is obtained from $\cN$ by cutting $(a, b)$, to construct $\cN'_A$, we need to determine which of the arcs $(p_1, p_b)$ and $(p_2, p_b)$ to delete from $\cN_A$.

Construct $\cN_{\{b, z\}}$ from $\cN_{\{a, b, z\}}$ in the usual way but with the following modification. Initially mark $p_a$ in $\cN_{\{a, b, z\}}$. When suppressing a marked vertex, mark its parent. The end result is $\cN_{\{b, z\}}$ with one of the parents of $p_b$ marked. The arc from the marked parent to $p_b$ corresponds to a path in $\cN_{\{a, b, z\}}$ from the marked parent to $p_b$ through $p_a$, and thus the arc we want to delete. On the other hand, we can also construct $\cN_{\{b, z\}}$ as the phylogenetic network exhibited by $\cN_{\blue{A=\{b, x, y\}}}$ on $\{b, z\}$. In doing this, mark the vertex $p_1$. If a marked vertex is suppressed, mark its parent. We again get $\cN_{\{b, z\}}$ with a parent of $p_b$ marked, and can compare our two marked parents. By Lemma~\ref{ret_construct_trinet}(iv), if they are the same vertex, $\cN'_A$ is constructed from $\cN_A$ by deleting the arc $(p_1, p_b)$, repeatedly deleting vertices of out-degree zero, and taking the full simplification of the resulting directed graph. Otherwise, by Lemma~\ref{ret_construct_trinet}(iv) again, $\cN'_A$ is constructed from $\cN_A$ by deleting the arc $(p_2, p_b)$, repeatedly deleting vertices of out-degree zero, and taking the full simplification of the resulting directed graph.

We conclude that the trinets exhibited by $\cN'$ (resp.\ $\cN'_0$) can be determined from the trinets exhibited by $\cN$ (resp.\ $\cN_0$). Since, \blue{up to isomorphism,} $Tn(\cN)=Tn(\cN_0)$ and there is no difference in the way $Tn(\cN')$ and $Tn(\cN'_0)$ are determined from $Tn(\cN)$ and $Tn(\cN_0)$, respectively, it follows that, up to isomorphism, $Tn(\cN_0')=Tn(\cN')$. Therefore, by the induction assumption, $\cN'\cong \cN_0'$. To construct $\cN$ and $\cN_0$ from $\cN'$ and $\cN'_0$, respectively, we need to realise $(a, b)$ as a reticulated cherry. The only way this can be done for $\cN'$ (and similarly for $\cN'_0$) is by subdividing the arcs into $a$ and $b$ with new vertices $p_a$ and $p_b$, and then adding an arc from $p_a$ to $p_b$. Hence $\cN\cong  \cN_0$, and this completes the proof of Theorem~\ref{main}(i).
\end{proof}

\subsection{Algorithm}

Let $\cN$ be an orchard network on $X$, where $|X|\geq 3$. The inductive proof of Theorem~\ref{main} implies a recursive algorithm that takes $X$ and $Tn(\cN)$ as its input and returns an orchard network $\cN_0$ isomorphic to $\cN$. \blue{Called} {\sc Construct Orchard}, we next describe this algorithm and give its running time. The correctness of {\sc Construct Orchard} is essentially established in the proof of Theorem~\ref{main}(i), and so it is omitted.

\begin{enumerate}[1.]
\item If $Tn(\cN)$ consists of a single trinet, \blue{and so $|X|=3$}, then return this trinet.

\item Else $Tn(\cN)$ contains at least two trinets, and so $|X|\ge 4$. Find elements $a, b\in X$ such that either $\{a, b\}$ is a cherry of every trinet in $Tn(\cN)$ whose leaf set contains both $a$ and $b$, or $(a, b)$ is a reticulated cherry of every trinet in $Tn(\cN)$ whose leaf set contains both $a$ and $b$.

\item If $\{a, b\}$ is a cherry of every trinet in $Tn(\cN)$ whose leaf set contains both $a$ and $b$, do the following:
\begin{enumerate}[3.1]
\item Let $Tn'(\cN)$ denote the set of trinets obtained from $Tn(\cN)$ by removing every trinet whose leaf set contains $b$.

\item Apply \textsc{Construct Orchard} to input $X'=X-\{b\}$ and $Tn'(\cN)$, and construct $\cN_0$ from the returned orchard network $\cN_0'$ by subdividing the arc directed into $a$ with a new vertex $p_a$, and adjoining a new leaf $b$ to $p_a$ via a new arc $(p_a, b)$.

\item Return $\cN_0$.
\end{enumerate}

\item Else $(a, b)$ is a reticulated cherry of every trinet in $Tn(\cN)$ whose leaf set contains both $a$ and $b$.
\begin{enumerate}[4.1]
\item Let $Tn'(\cN)$ denote the set of trinets obtained from $Tn(\cN)$ by replacing each trinet $\cN_A\in Tn(\cN)$ in which $b\in A$ with the trinet $\cN'_A$ constructed as follows:
\begin{enumerate}[4.{1}.1]
\item If $a\in A$, construct $\cN'_A$ from $\cN_A$ by deleting the reticulation arc of $(a, b)$ and suppressing the two resulting vertices of in-degree one and out-degree one.

\item Else $A=\{b, x, y\}$ for some distinct $x, y\in X-\{a, b\}$. Set $p_x$ (resp.\ $p_y$) to be the parent of $a$ in $\cN_{\{a, b, x\}}$ (resp.\ $\cN_{\{a, b, y\}}$), and set $p'_x$ (resp.\ $p'_y$) to be the parent of $b$ in $\cN_{\{a, b, x\}}$ (resp.\ $\cN_{\{a, b, y\}}$). Create a new directed graph $G_x$ (resp.\ $G_y$) from $\cN_{\{a, b, x\}}$ (resp.\ $\cN_{\{a, b, y\}}$) by deleting $a$ and taking the full simplification. Each time $p_x$ (resp.\ $p_y$) is suppressed during this process, set $p_x$ (resp.\ $p_y$) to be the parent of the suppressed vertex instead.

\begin{enumerate}[4.{1}.{2}.1]
\item If neither $p'_x$ nor $p'_y$ is a vertex of $G_x$ and $G_y$, respectively, then choose $\cN'_A$ to be $\cN_A$.

\item Else there is an element $z\in \{x, y\}$ such that $p'_z$ is a vertex of $G_z$. Let $\{z, z'\}=\{x, y\}$. Denote the parent of $b$ in $\cN_A$ by $p_b$, and let $p_1$ and $p_2$ be the parents of $p_b$ in $\cN_A$. Create a new directed graph $G'_z$ from $\cN_A$ by deleting every vertex of $\cN_A$ whose only leaf descendant is $z'$ and taking the full simplification. Each time $p_1$ is suppressed during this process, set $p_1$ to be the parent of the suppressed vertex instead.

\item Compare $p_z$ and $p_1$ in the isomorphic directed graphs $G_z$ and $G'_z$. If $p_z$ and $p_1$ are the same vertex, then construct $\cN'_A$ from $\cN_A$ by deleting $(p_1, p_b)$, repeatedly deleting vertices of out-degree zero, and \blue{then} taking the full simplification. Else construct $\cN'_A$ from $\cN_A$ by deleting $(p_2, p_b)$, repeatedly deleting vertices of out-degree zero, and \blue{then} taking the full simplification.
\end{enumerate}
\end{enumerate}

\item Apply \textsc{Construct Orchard} to input $X$ and $Tn'(\cN)$, and construct $\cN_0$ from $\cN_0'$ by subdividing the arcs directed into $a$ and $b$ with new vertices $p_a$ and $p_b$, respectively, and adjoining $p_a$ and $p_b$ via a new arc $(p_a, p_b)$.

\item Return $\cN_0$.
\end{enumerate}
\end{enumerate}

We now consider the running time of {\sc Construct Orchard}.

\begin{proof}[Proof of Theorem~\ref{main}(ii)]
The algorithm takes as input a set $X$ and the set $Tn(\cN)$ of trinets of an orchard network $\cN$ on $X$. If $Tn(\cN)$ consists of a single trinet, then {\sc Construct Orchard} runs in constant time. If $Tn(\cN)$ contains at least two trinets, and so $|X|\ge 4$, the algorithm begins by finding a $2$-element subset $\{a, b\}$ of $X$ such that either $\{a, b\}$ is a cherry of every trinet in $Tn(\cN)$ whose leaf set contains both $a$ and $b$, or $(a, b)$ is a reticulated cherry of every trinet in $Tn(\cN)$ whose leaf set contains both $a$ and $b$. There at most $|X|^2$ choices for a $2$-element subset of $X$. Since there are $O(|X|^3)$ trinets and deciding if $\{a, b\}$ is a cherry, or $(a, b)$ or $(b, a)$ is a reticulated cherry of a trinet takes constant time, the running time of Step~2 of {\sc Construct Orchard} takes $O(|X|^5)$ time. Once such a $2$-element subset is found, the algorithm constructs a new set $Tn'(\cN)$ of trinets from $Tn(\cN)$. In the worst possible instance, the longest running part of this process is when, $(a, b)$ say, is a reticulated cherry of every trinet of $Tn(\cN)$ whose leaf set contains both $a$ and $b$, and Step~4.1 is invoked.

Let $V$ denote the vertex set of $\cN$. Now, $Tn'(\cN)$ is obtained from $Tn(\cN)$ by modifying the trinets $\cN_A$ of $Tn(\cN)$ whose leaf set contains $b$. Thus there are at most $|X|^2$ such trinets to consider. In terms of running time, the longest part of Step~4.1 is when $A=\{b, x, y\}$, where $a\not\in \{x, y\}$, and Step~4.1.2 is invoked. The directed graphs $G_x$ and $G_y$ take $O(|V|^2)$ time to construct from $\cN_{\{a, b, x\}}$ and $\cN_{\{a, b, y\}}$, respectively. After that, Step~4.1.2.1 takes constant time. If Step~4.1.2.2 is called, \blue{determining} $z$ takes constant time and constructing $G'_z$ from $\cN_A$, where $z\in \{x, y\}$, takes $O(|V|^2)$ time. In Step~4.1.2.3, the directed graphs $G_z$ and $G'_z$ are compared to decide whether $p_z$ and $p_1$ are the same vertex. This comparison takes $O(|V|^2)$ time and, regardless of the decision, the resulting construction of $\cN'_A$ takes $O(|V|^2)$ time. Hence the running time to complete Step~4.1 is $O(|X|^2|V|^2)$.

With Step~4.1 completed, Steps~4.2 and~4.3 \blue{each} take constant time. It follows that each iteration takes $O(|X|^5+|X|^2|V|^2)$ time. When recursing, the input to the recursive call is either a set $X'=X-\{b\}$ and a set $Tn'(\cN)$ of trinets of an orchard network on $|X|-1$ leaves and $r$ reticulations, or a set $X$ and a set $Tn'(\cN)$ of trinets of an orchard network on $|X|$ leaves and $r-1$ reticulations, where $r$ is the number of reticulations of $\cN$. Therefore the total number of iterations is $O(|X|+r)$. \blue{Since $|V|=2(|X|+r)-1$~\cite{mcd15}, it follows that the total number of iterations is} $O(|V|)$. Hence {\sc Construct Orchard} completes in
$$O(|V|(|X|^5+|X|^2|V|^2))$$
time, that is, in $O(|V|^6)$ time as $|X|\le |V|$. This completes the proof of Theorem~\ref{main}(ii).
\end{proof}

\section{An Example}
\label{example}

In this section, we show that the largest value of $k$ such that all recoverable level-$k$ phylogenetic networks $\cN$ are encoded by $Tn(\cN)$ is at most~$3$. Consider the two phylogenetic networks $\cN_1$ and $\cN_2$ on $\{x_1, x_2, x_3, x_4, x_5\}$ shown in Fig.~\ref{counterexample}. Both $\cN_1$ and $\cN_2$ are recoverable, but $\cN_1$ is not isomorphic to $\cN_2$ as the vertex which is an ancestor of $x_1$ and $x_2$ but no other leaves is a descendant of the parent of $x_5$ in $\cN_1$, but is not a descendant of the parent of $x_5$ in $\cN_2$.

\begin{figure}
	\centering
	\begin{subfigure} {0.5\textwidth}
		\centering
		\begin{tikzpicture}
		\draw (0,0) -- (-1,-1.5);
		\draw (-0.5,-0.75) -- (0,-1.5);
		\draw (-1,-1.5) -- (-1,-3);
		\draw (0,0) -- (1,-1.5);
		\draw (-1,-1.5) -- (1,-3);
		\draw (1,-1.5) -- (1,-3);
		\draw (-1,-3) -- (-1,-4);
		\draw (1,-3) -- (1,-4);
		
		\filldraw[white] (0,-2.25) circle (2pt);
		
		\draw (1,-1.5) -- (-1,-3);
		
		\filldraw[black] (0,0) circle (2pt);
		\filldraw[black] (-1,-1.5) circle (2pt);
		\filldraw[black] (-0.5,-0.75) circle (2pt);
		\filldraw[black] (0,-1.5) circle (2pt) node[label=below:$x_5$] {};
		\filldraw[black] (-1,-3) circle (2pt);
		\filldraw[black] (1,-3) circle (2pt);
		\filldraw[black] (1,-1.5) circle (2pt);
		\filldraw[black] (-1,-4) circle (2pt) node [label=below:$x_i$] {};
		\filldraw[black] (1,-4) circle (2pt) node [label=below:$x_j$] {};
		\end{tikzpicture}
		\subcaption{Trinet exhibited by $\cN_1$ on $\{x_i, x_j, x_5\}$ for all $\{i, j\}\subseteq \{1, 2, 3, 4\}$.}
	\end{subfigure}
	\begin{subfigure} {0.45\textwidth}
		\centering
		\begin{tikzpicture}
		\draw (0,0) -- (-1,-1.5);
		\draw (0.5,-0.75) -- (0,-1.5);
		\draw (-1,-1.5) -- (-1,-3);
		\draw (0,0) -- (1,-1.5);
		\draw (-1,-1.5) -- (1,-3);
		\draw (1,-1.5) -- (1,-3);
		\draw (-1,-3) -- (-1,-4);
		\draw (1,-3) -- (1,-4);
		
		\filldraw[white] (0,-2.25) circle (2pt);
		
		\draw (1,-1.5) -- (-1,-3);
		
		\filldraw[black] (0,0) circle (2pt);
		\filldraw[black] (-1,-1.5) circle (2pt);
		\filldraw[black] (0.5,-0.75) circle (2pt);
		\filldraw[black] (0,-1.5) circle (2pt) node[label=below:$x_5$] {};
		\filldraw[black] (-1,-3) circle (2pt);
		\filldraw[black] (1,-3) circle (2pt);
		\filldraw[black] (1,-1.5) circle (2pt);
		\filldraw[black] (-1,-4) circle (2pt) node [label=below:$x_i$] {};
		\filldraw[black] (1,-4) circle (2pt) node [label=below:$x_j$] {};
		\end{tikzpicture}
		\subcaption{Trinet exhibited by $\cN_2$ on $\{x_i, x_j, x_5\}$ for all $\{i, j\}\subseteq \{1, 2, 3, 4\}$.}
	\end{subfigure}
	\caption{The phylogenetic networks $\cN_1$ and $\cN_2$ as shown in Fig.~\ref{counterexample} exhibit, up to isomorphism, the same trinets on $\{x_i, x_j, x_5\}$, where $\{i, j\}\subseteq \{1, 2, 3, 4\}$.}
	\label{trinets_same_fig}
\end{figure}

Up to isomorphism, the path graphs of $\cN_1$ and $\cN_2$ on $\{x_1, x_2, x_3, x_4\}$ are the same. Therefore, any trinet of $Tn(\cN_1)$ and $Tn(\cN_2)$ on the same leaf set not containing $x_5$ are isomorphic. Furthermore, every other trinet of $\cN_1$ and $\cN_2$ is isomorphic to the trinet shown in Fig.~\ref{trinets_same_fig}(A) and~(B), respectively, where $\{i, j\}\subseteq \{1, 2, 3, 4\}$. Since the trinets in this figure are isomorphic, it follows that $\cN_1$ is not encoded by $Tn(\cN_1)$.

We end this section, \blue{and the paper}, with some remarks. For a non-negative integer $k$, a phylogenetic network $\cN$ is {\em level-$k$} if each biconnected component of $\cN$ has at most $k$ reticulations. It is shown in~\cite{ier14} that all recoverable (binary) level-$2$ networks are encoded by their sets of trinets. The authors comment that the approach taken to establish this \blue{uniqueness} result does not extend to level-$k$ networks, where $k\ge 4$. Curiously, the counterexample consists of two level-$4$ networks. This raises the question of whether a recoverable level-$3$ network is encoded by the set of trinets it exhibits.

\end{document}